\documentclass[11pt,oneside,reqno]{amsart}
\usepackage[utf8]{inputenc}
\usepackage{enumerate}

\numberwithin{equation}{section}
\theoremstyle{definition}
\newtheorem{Definition}{Definition}[section]
\newtheorem{Example}[Definition]{Example}
\newtheorem{Remark}[Definition]{Remark}

\theoremstyle{plain}
\newtheorem{Theorem}[Definition]{Theorem}
\newtheorem{Proposition}[Definition]{Proposition}
\newtheorem{Corollary}[Definition]{Corollary}
\newtheorem{Lemma}[Definition]{Lemma}

\newtheorem{Question}[Definition]{Question}

\usepackage{amssymb}   
\usepackage{mathrsfs}  
\usepackage{eucal}     

\usepackage{geometry}

\usepackage[colorlinks=true,allcolors=blue]{hyperref}

\usepackage{tikz}
\usetikzlibrary{arrows, matrix}
\usepackage{tikz-cd}

\newcommand{\be}{\beta}

\newcommand{\Ga}{\Gamma}

\newcommand{\ep}{\varepsilon}
\renewcommand{\epsilon}{\varepsilon}
\newcommand{\la}{\lambda}
\newcommand{\si}{\sigma}

\newcommand{\Z}{\mathbb{Z}}

\newcommand{\C}{\mathbb{C}}

\newcommand{\I}{\mathbb{I}}

\renewcommand{\frak}{\mathfrak}
\newcommand{\Fgl}{\mathfrak{gl}}
\newcommand{\Fsl}{\mathfrak{sl}}
\newcommand{\Fg}{\mathfrak{g}}
\newcommand{\Fh}{\mathfrak{h}}
\newcommand{\Fk}{\mathfrak{k}}
\newcommand{\Fm}{\mathfrak{m}}
\newcommand{\Fn}{\mathfrak{n}}

\newcommand{\CA}{\mathcal{A}}

\newcommand{\SMod}{\mathsf{Mod}}
\newcommand{\SVec}{\mathsf{Vec}}

\DeclareMathOperator{\ann}{ann}

\DeclareMathOperator*{\colim}{colim}

\DeclareMathOperator{\End}{End}
\DeclareMathOperator{\Ext}{Ext}
\DeclareMathOperator{\Fix}{Fix}
\DeclareMathOperator{\Hom}{Hom}

\newcommand{\cotimes}[1]{\mathbin{\mathop{\otimes}\limits_{#1}}}

\newcommand{\Bmod}{{}_{\mathcal{B}^{\rm op}}\mathsf{Mod}}

\newcommand{\HC}{\mathsf{GenWt}(A,\mathbb{X})}

\title[Subcategories of Module Categories via Restricted Yoneda Embeddings]{Subcategories of Module Categories\\ via Restricted Yoneda Embeddings}
\date{}

\author{Dylan Fillmore} \address{Shenzhen International Center for Mathematics, Southern University of Science and Technology, Shenzhen, China} 
  \email{fillmore@sustech.edu.cn}
  \urladdr{http://dfillmore.com}
\author{Jonas T. Hartwig}
\address{Department of Mathematics, Iowa State University, Ames IA 50011, USA}
\email{jth@iastate.edu}
\urladdr{http://jthartwig.net}
\thanks{J.T.H. is supported in part by the United States Army Research Office grant W911NF-24-1-0058.}

\begin{document}

\begin{abstract}
We propose a framework for producing interesting subcategories of the category ${}_A\SMod$ of left $A$-modules, where $A$ is an associative algebra over a field $k$. 
The construction is based on the composition, $Y$, of the Yoneda embedding of ${}_A\SMod$ with a restriction to certain subcategories $\mathcal{B}\subset {}_A\SMod$, typically consisting of cyclic modules. We describe the subcategories on which $Y$ provides an equivalence of categories. This also provides a way to understand the subcategories of ${}_A\SMod$ that arise this way. Many well-known categories are obtained in this way, including categories of weight modules and Harish-Chandra modules with respect to a subalgebra $\Gamma$ of $A$. In other special cases the equivalence involves modules over the Mickelsson step algebra associated to a reductive pair of Lie algebras.
\end{abstract}

\maketitle

\section{Introduction}
In a celebrated paper \cite{Blo1981}, Block classified all (including infinite-dimensional) irreducible representations of the Lie algebra $\Fsl_2(k)$ over an arbitrary field $k$.
Already the classification of irreducible representations of $\Fsl_3(k)$ is known to be a \emph{wild problem}, in the sense of the tame/wild dichotomy of Drozd \cite{Dro1980,Iov2018}. This means that a complete classification of all irreducible representations of $\Fsl_3(k)$ would imply a complete classification of all finite-dimensional indecomposable representations of \emph{all finite-dimensional associative algebras}.
Such problems are considered ill advised to pursue.

A way out is of course to restrict attention to subcategories of the category ${}_A\SMod$ of left $A$-modules, for a given associative algebra $A$ (or for $A$ belonging to some class of similar algebras). The following question now arises:

\begin{Question}\label{question}
    How do we identify subcategories of ${}_A\SMod$ such that progress towards understanding their structure is possible?
\end{Question}

An obvious choice is the category of finite-dimensional modules. However, for many algebras arising in Lie theory, this subcategory is semisimple, and irreducible modules have been classified long ago.

A larger and richer example, in the case when $A$ is the enveloping algebra $U(\Fg)$ of a complex finite-dimensional reductive Lie algebra $\Fg$, is the famous BGG category $\mathcal{O}$. Although irreducibles are simply parametrized by the dual of the Cartan, it is not a semisimple category. Category $\mathcal{O}$ is further a subcategory of the category of \emph{weight modules}. The classification of irreducible weight modules for $U(\Fg)$ has been achieved by Fernando \cite{Fer1990} and Mathieu \cite{Mat2000}.

Categories of weight modules play an important role for other algebras, such as \emph{(twisted) generalized Weyl algebras} \cite{Bav1992,MazTur1999,Gad2025}. This is a surprisingly diverse class of associative algebras $A$ which contain a distinguished (typically commutative) subalgebra $R$. A key reason for this is that $R$ is a normal subalgebra of $A$. This means there is a set of generators $\{a_i\}_i$ for $A$ such that $a_iR=Ra_i$.

There are important algebra-subalgebra pairs $(A,\Gamma)$ where the subalgebra $\Gamma$ is not normal, including $A=U(\Fgl_n)$ and $\Gamma$ the Gelfand-Tsetlin subalgebra, by definition generated by $Z_1\cup\cdots\Z_n$, $Z_k:=Z(U(\Fgl_k))\subset A$. Instead, $\Gamma$ is \emph{quasi-central} in $A$, in the terminology of \cite{DFO1994}. Then one is forced to study the category of generalized weight modules with respect to $\Ga$. A weaker but better behaved condition that can replace the quasi-central property was found in \cite{Fil2024}. The condition is that $A/A\Fm$ itself is a generalized weight module for every (or some) finite codimension maximal ideal $\Fm\subset\Gamma$.
This observation may be regarded as a starting point for the present paper.

Our proposed method for constructing subcategories of ${}_A\SMod$,  addressing Question \ref{question}, is the following:

Let $A$ be an associative algebra over a field $k$. The input of the construction is a subcategory $\mathcal{B}$ of ${}_A\SMod$ whose objects (or maps out of them) are well understood. In most of our examples, $\mathcal{B}$ consists of certain cyclic modules. Next, consider the composition of two functors:
\begin{equation}\label{eq:intro-composition}
    {}_A\SMod \longrightarrow [{}_A\SMod^{\rm op},\, \SVec_k]\longrightarrow [\mathcal{B}^{\rm op},\,\SVec_k]
\end{equation}
The first functor is the Yoneda embedding, $V\mapsto \Hom_A(-,V)$. The notation $[\mathcal{C},\mathcal{D}]$ stands for the category of $k$-linear functors between $k$-linear categories $\mathcal{C}$ and $\mathcal{D}$. The second functor in \eqref{eq:intro-composition} is the pullback through the inclusion functor $\mathcal{B}\to {}_A\SMod$. In Theorem \ref{thm:adj}, we show that this composition has a left adjoint. As with any adjunction, it restricts to an equivalence of categories. The output of this construction is therefore two categories. The first, denoted $\Fix_\ep$, is a subcategory of ${}_A\SMod$. We claim that this is an interesting subcategory. By definition, $\Fix_\ep$ is the full subcategory of ${}_A\SMod$ consisting of left $A$-modules $V$ for which the $V$-component $\ep_V$ of the counit of the adjunction is an isomorphism. The second, denoted $\Fix_\eta$, is a subcategory of the functor category $[\mathcal{B},\,\SVec_k]$. Our perspective is that $\Fix_\eta$ is the ``answer'' to the problem of ``computing'' the category $\Fix_\ep$.
Because the ``answer'' is built-in to the construction, we view this as a way to address Question \ref{question}.

A slightly larger subcategory than $\Fix_\ep$, that also seems to be interesting, is the full subcategory of left $A$-modules $V$ such that $\ep_V$ is an epimorphism, but not necessarily an isomorphism. Although the construction does not then provide an equivalence, it comes with a functor that can be used to study this larger category.

A future direction of research would be to address for which choices of input categories $\mathcal{B}$, the corresponding subcategory $\Fix_\ep$ of ${}A_\SMod$ is tame.

Let us briefly summarize the contents of the paper. Section \ref{sec:general} contains the general results concerning the adjunction and technical details about epic and monic transformations. We supply several detailed examples in Section \ref{sec:examples} that constitute support for thesis that the $\Fix_\ep$'s form an interesting class of subcategories.

To put our main results in context, we end this introduction by providing an overview of several constructions from ring theory and representation theory that will turn out to be special cases of the general construction.

\subsection{``Single Object'' Settings}
\label{sec:intro-single-object}

\subsubsection{Quantum Hamiltonian Reductions}
\label{sec:quantum-hamiltonian-reduction}
In the theory of deformation quantization and Poisson geometry, there is a notion of \emph{quantum Hamiltonian reductions}. For brief introductions on this topic, see \cite{YiSun}\cite{LosevNotes}. Important examples of algebras obtained by quantum Hamiltonian reduction include \emph{finite $W$-algebras} $U(\Fg,e)$. These arise as the Zhu algebra of certain vertex operator algebras called $\mathcal{W}$-algebras. They are also quantizations of Slodowy slices. Here $\Fg$ is a semisimple Lie algebra and $e\in \Fg$ is a nilpotent element. Briefly, out of $e$ one constructs a nilpotent subalgebra $\Fm\subseteq \Fg$ and a Lie algebra homomorphism $\chi:\Fm\to\C$ (see e.g. \cite{GanGin2002} for details). Let $\Fm_\chi=\ker\chi$ and
put $I=U(\Fg)\Fm_\chi$. The finite $W$-algebra is \[U(\Fg,e)=N_U(I)/I.\] When $e=0$, $U(\Fg,e)=U(\Fg)$ and when $e$ is of principal type (all $1$'s on the superdiagonal and zeroes elsewhere, for $\Fg=\Fsl_n$) then $U(\Fg,e)=Z(\Fg)$, the center of $U(\Fg)$.

It was shown in \cite{Tik2012} that quantum Hamiltonian reductions of central quotients of infinitesimal Hecke algebras (deformations of $U(\Fsl_2\rtimes\C^2)$) are noncommutative deformations of Kleinian singularities of type $D$. 

\subsubsection{Mickelsson Step Algebras and Reduction Algebras}
\label{sec:intro-reduction-algebras}

In 1973, Mickelsson \cite{Mic1973} introduced the notion of step algebra $S(\Fg,\Fk)$. It is constructed from a finite-dimensional complex Lie algebra $\Fg$, with a subalgebra $\Fk$ which is reductive in $\Fg$ (i.e. $\Fg$ is $\mathrm{ad }\Fk$-semisimple). In particular $\Fk$ is reductive and one picks a triangular decomposition $\Fk=\Fn_-\oplus\Fh\oplus\Fn_+$ and let $I=U(\Fg)\Fn_+$ be the left ideal in $U(\Fg)$ generated by $\Fn_+$. Mickelsson's step algebra is defined as \[S(\Fg,\Fk)=N_{U(\Fg)}(I)/I,\] the eigenring of $I$ (in ring theory terminology).
More generally one can replace $U(\Fg)$ by any associative algebra $A$ equipped with an algebra homomorphism $U(\Fk)\to A$, and put $S(A,\Fk)=N_A(A\Fn_+)/A\Fn_+$. These, and closely related localizations, have been termed \emph{reduction algebras}, see for example \cite{KhoOgi2008}. There are also quantum, super, and Kac-Moody generalizations of this story.

\subsubsection{Normalizer Algebras}
\label{sec:normalizer-of-a-left-ideal}
The following setting generalizes the previous two cases.
Let $k$ be a field, and $A$ be an associative $k$-algebra. Let ${}_A\SMod$ denote the category of left $A$-modules.
Let $I$ be a left ideal of $A$. Let $N_A(I)=\{a\in A\mid Ia\subset I\}$ be the normalizer of $I$ in $A$, and let 
\begin{equation}
    \CA(I)=N_A(I)/I.
\end{equation}
In ring theory, $N_A(I)$ is known as the \emph{idealizer ring} \cite[§1.1.11--13; §5.5.1--11]{McCRob2001}. The quotient $N_A(I)/I$ is then called the \emph{eigenring} of $I$.
There are two additional descriptions of $\CA(I)$. Let $(-)^I:{}_A\SMod\to {}_{\CA(I)}\SMod$ denote the covariant functor sending a left $A$-module $V$ to the left $\CA(I)$-module
\[V^I=\{v\in V\mid Iv=0\}\]
of \emph{$I$-invariants}. Then we have
\begin{equation}
    \CA(I)=(A/I)^I.
\end{equation}
The third description of the algebra $\CA(I)$ is as the opposite endomorphism algebra of the left $A$-module $A/I$:
\begin{equation}\label{eq:normalizer-iso}
    \CA(I)\cong \End_A(A/I)^{\rm op}.
\end{equation}
The isomorphism sends $a+I\in\CA(I)$ to the endomorphism $\phi_a$ sending $b+I$ to $ba+I$.
We also observe here that the functor $(-)^I$ is a right adjoint: For any $V\in {}_{\CA(I)}\SMod$ and $W\in {}_A\SMod$ we have a natural isomorphism
\begin{equation}
    \Hom_A\big(A/I\otimes_{N_A(I)} V,\, W)\cong \Hom_{\CA(I)}(V,\,W^I).
\end{equation}
This follows from tensor hom adjunction using the fact that $W^I\cong \Hom_A(A/I,W)$.

\subsection{``Many Object'' Settings}
\subsubsection{Generalized Weyl Algebras and their Weight Modules}
Bavula \cite{Bav1992} introduced and studied a class of algebras called \emph{generalized Weyl algebras}. These were further generalized in \cite{MazTur1999} to so-called \emph{twisted generalized Weyl algebras}. See also \cite{Gad2025} for a survey on these and related algebras.
In \cite{MazPonTur2003}, among other things, it was shown that simple weight modules (with finite-dimensional weight spaces $V_\Fm$) $V$ with $V_\Fm\neq 0$ (for fixed $\Fm$) are in bijective correspondence with finite-dimensional simple modules over the so-called \emph{cyclic subalgebra} $C(\Fm)$. This is in fact a special case of results by Lepowsky and McCollum from 1973 \cite{LepMcC1973}.

\subsubsection{Zhelobenko's Category}
In \cite{NerKho2009}, the authors cite work from 1969 by Zhelobenko \cite{Zhe1969} in which he considered the following category. 
Let $U(\Fg)$ be the enveloping algebra of a Lie algebra $\Fg$ and $U(\Fk)$ the enveloping algebra of a semisimple Lie subalgebra $\Fk\subseteq\Fg$. For dominant integral weight $\la\in\Fh^\ast$, let $I_\la\subseteq U(\Fk)$ be the annihilator of the finite-dimensional irreducible highest weight $\Fk$-module $L_\la$ with highest weight $\la$. Next, by $U_{\la\mu}$ denote the set of all $u\in U(\Fg)$ such that $I_\la u\subseteq I_\mu$. We have $U_{\la\mu}U_{\mu\nu}\subseteq U_{\la\nu}$. Therefore, we get a category whose objects are dominant weights $\lambda$ for $\Fk$ and whose morphisms are elements $u\in U_{\la\mu}$.
For a finite-dimensional representation $V$ of $\Fg$ we consider its restriction to $\Fk$. Since $\Fk$ is semisimple, $V$ decomposes into a direct sum $V=\bigoplus_\la V(\la)$, where $V(\la)$ is an isotypical component, i.e. $V(\la)=L_\la^{m_\la}$ for certain multiplicity $m_\la$. It is easy to check that $u\in U_{\la\mu}$ sends $V(\mu)$ to $V(\la)$. In other words, we get representation of our category.

The authors of \cite{NerKho2009} say that this structure was also considered in Diximier’s book \cite[Ch.~9]{Dix1996} with a reference to Lepowsky and McCollum \cite{LepMcC1973}, though the structure is quite natural and perhaps had been introduced earlier.

\subsubsection{Normalizer Categories}

The previous two examples fit into the following picture.
Let $X$ be a set of left ideals of $A$. For $I,J\in X$, 
define
\begin{equation}
    \CA(I,J)=(A/I)^J=\{a+I\in A/I\mid Ja\subset I\}.
\end{equation}
Then multiplication in $A$ induces a well-defined map
\begin{equation}
    \CA(J,K)\otimes_k \CA(I,J)\to\CA(I,K),\quad (a+J)\otimes (b+I)\mapsto ab+I
\end{equation}
(Well-defined: if $b\in I$ then $ab+I\subset I$; if $a\in J$ then $ab+I\subset Jb+I\subset I$. Correct codomain: $Kab\subset Jb\subset I$ whenever $a+J\in\CA(J,K)$ and $b+I\in\CA(I,J)$.)

In other words, we have a category $\CA=\CA_X$ with objects being the elements of $X$ and morphism set $\CA(I,J)$ as defined above. Notice that when $X$ is a singleton, say $X=\{I\}$, then the category $\CA$ coincides with the algebra $\CA(I)$ from Section \ref{sec:normalizer-of-a-left-ideal}, viewed as a category with one object. Generalizing the singleton case \eqref{eq:normalizer-iso}, we have a natural isomorphism
\begin{equation}\label{eq:endomorphism-description}
    \CA(I,J)\cong\Hom_A(A/J,A/I)
\end{equation}
given by sending $a+I\in\CA(I,J)$ to $\phi_a$ given by $\phi_a(b+J)=ba+I$.

As in the single object case, we can turn $A$-modules into $\CA$-modules: Let $V$ be a left $A$-module. Then 
there is a functor $\CA\to {}_k\SMod$ sending an object $I$ to $V^I$ and a morphism $a+I\in\CA(I,J)$ to the $k$-linear map $V^I\to V^J$ given by $v\mapsto av$ (here $av\in V^J$ since $Jav\subset Iv=0$). This is a covariant functor ${}_A\SMod\to{}_\CA\SMod$.
This can be recast in more categorical language as follows.
Let $\mathcal{B}$ be the full subcategory of ${}_A\SMod$ with objects $A/I$ for $I \in X$. Then $\mathcal{A}\cong\mathcal{B}^{\mathrm{op}}$.

\subsubsection{Harish-Chandra Modules}

Let $U$ be the universal enveloping algebra of the general linear Lie algebra $\mathfrak{gl}_n(\C)$. The category of finite-dimensional representations of $U$ is well-understood, as is the larger category $\mathcal{O}$, and still larger category of all weight modules with respect to the Cartan subalgebra.
An even larger category is that of \emph{Harish-Chandra} (or \emph{Gelfand-Tsetlin}) modules over $U$. These are modules on which a certain maximal commutative subalgebra $\Ga\subset U$ acts locally finitely.
This category was studied in \cite{DFO1994} and has led to the definition of \emph{Galois algebras}, see \cite{Fut2018} and references therein. A key feature of these algebras $A$ is that they, like $U$, contain a large commutative subalgebra $\Ga$, with respect to which one can define Harish-Chandra categories of modules.

Harish-Chandra modules do not fit into the normalizer category picture from the previous subsection. A more general framework is needed to capture them. Topologically enriched approaches were considered in \cite{DFO1994} and more recently in \cite{Fil2024}.

\section{Restricted Yoneda Embeddings and Realizations}
\label{sec:general}

Fix a field $k$. We assume all categories are $k$-linear, i.e. enriched over the category $\SVec$ of $k$-vector spaces. Functors are assumed to be $k$-linear on hom spaces.
Let $A$ be an associative $k$-algebra.

\begin{Definition}[{\cite{Kelly}}]
Let $\mathcal{B}$ be a category, and $i:\mathcal{B}\to {}_A\SMod$ be a linear functor.
The \emph{restricted Yoneda embedding} is the Yoneda embedding of ${}_A\SMod$ followed by pullback through $i$:
\[
Y:{}_A\SMod \to \Bmod, \quad V\mapsto \Hom_{A}(i(\cdot),V).
\]
\end{Definition}

\begin{Theorem}\label{thm:adj}
There is a left adjoint $|\cdot |$ to the restricted Yoneda embedding $Y$ above. Explicitly, for $F,G \in \Bmod$ and $\theta: F \Rightarrow G$, $|F|$ is the coend
\begin{equation}
        |F|= \int^{B \in \mathcal{B}} i(B) \otimes_k F(B), 
\end{equation}
and $|\theta|$ is the map
\begin{equation}
        |\theta|= \int^{B \in \mathcal{B}} i(B) \otimes_k \theta_B. 
\end{equation}
\end{Theorem}

\begin{proof}
Let $V$ be a left $A$-module and $F$ be a linear functor from $\Bmod$ to $\SVec$. Then we have:
\begin{align*}
    \Hom_A(|F|,V) &= \Hom_A(\int^{B \in \mathcal{B}} i(B) \otimes_k F(B), V)\\
    &\cong \int_{B \in \mathcal{B}}\Hom_A( i(B) \otimes_k F(B), V)\\
    &\cong \int_{B \in \mathcal{B}}\Hom_k(F(B), \Hom_A(i(B),V))\\
    &= \int_{B \in \mathcal{B}}\Hom_k(F(B), Y(V)(B))\\
    &\cong \mathrm{Nat}(F, Y(V)).
\end{align*}
\end{proof}

\begin{Definition}
    $|F|$ is called the \emph{realization} of $F$.
\end{Definition}

\begin{Remark}
    Given $F \in \Bmod$, a more concrete description of $|F|$ is given by $$|F| =\Big( \bigoplus_{B \in \mathcal{B}} i(B) \otimes_k F(B) \Big)\Big/\langle f(b) \otimes v - b \otimes F(f)(v)\rangle_{f \in \text{Mor}(\mathcal{B})}.$$
    Alternatively, consider the natural structures of $\bigoplus_{B \in \mathcal{B}} i(B)$ and $\bigoplus_{C \in \mathcal{B}} F(C)$ as right and left modules, respectively, over the category algebra $k\mathcal{B}^{\rm op}$. Then
    \[|F| = \Big(\bigoplus_{B \in \mathcal{B}} i(B)\Big) \cotimes{k\mathcal{B}^{\rm op}} \Big(\bigoplus_{C \in \mathcal{B}} F(C)\Big).\]
\end{Remark}

\begin{Definition}
    Let $\eta,\epsilon$ be the unit and counit of the adjunction $|\cdot| \dashv Y$. We denote by
    \begin{enumerate}[{\rm (i)}]
        \item $\Fix_\epsilon=\{V \in {}_A\SMod: \epsilon_V \text{ is an isomorphism}\}$,\footnote{Unless otherwise stated, this type of equality is understood to mean ``$\Fix_\epsilon$ is the full subcategory of ${}_A\SMod$ with objects $\{V \in {}_A\SMod: \epsilon_V \text{ is an isomorphism}\}$''.}
        \item $\Fix_\eta=\{F \in \Bmod: \eta_F \text{ is an isomorphism}\}$.
    \end{enumerate}
\end{Definition}

\begin{Proposition}\label{prp:fix}
    We have an equivalence of categories $\Fix_\epsilon \cong \Fix_\eta$.
\end{Proposition}

\begin{proof}
    By basic category theory, any adjunction restricts to an equivalence on these subcategories.
\end{proof}

\begin{Remark}
It is the categories $\Fix_\ep$, for various choices of $(\mathcal{B},i)$ that we identify as interesting subcategories of ${}_A\SMod$ to study. Intuitively, we regard $\Fix_\eta$ as the ``answer'' to the problem of ``computing'' the category $\Fix_\ep$.
In the next section, we consider various choices of full subcategories $\mathcal{B}$ of ${}_A\SMod$ and describe the categories $\Fix_\epsilon$ and $\Fix_\eta$. 
\end{Remark}

\begin{Remark}\label{rem:pointwise}
We will be attempting to characterize when $\eta$ and $\epsilon$ are monomorphisms or epimorphisms. We note that the natural transformations we consider will be morphisms in one of the functor categories: $[\mathcal{B}^{\rm op}, \SVec], [{}_A\SMod, {}_A\SMod], [{}_{\mathcal{B}^{\rm op}}\SMod, {}_{\mathcal{B}^{\rm op}}\SMod]$. Since ${}_A\SMod$ and $\SVec$ are (co)complete, $[\mathcal{B}^{\rm op}, \SVec]$, $[{}_A\SMod, {}_A\SMod]$, and $[{}_{\mathcal{B}^{\rm op}}\SMod, {}_{\mathcal{B}^{\rm op}}\SMod]$ will also be (co)complete, with (co)limits being computed pointwise. Furthermore, $[\mathcal{B}^{\rm op}, \SVec]$, $[{}_A\SMod, {}_A\SMod]$, $[{}_{\mathcal{B}^{\rm op}}\SMod, {}_{\mathcal{B}^{\rm op}}\SMod]$ are abelian categories. Thus a natural transformation is a monomorphism, epimorphism, or isomorphism exactly when it is pointwise so. When dealing with categories besides these, we will include the qualifier ``pointwise'' if there is any ambiguity.
\end{Remark}

\begin{Lemma}\label{lem: unitfact}
    Let $L: \mathcal{C} \rightarrow \mathcal{D}$ and $R: \mathcal{D} \rightarrow \mathcal{C}$ be an adjunction $L \dashv R$ with unit $\eta$ and counit $\epsilon$. Suppose $\eta$ is pointwise monic and $\epsilon L$ is pointwise split monic. Then $\eta$ is pointwise epic.
\end{Lemma}
\begin{proof}
    Suppose $P,Q \in \mathcal{C}$ and $f,g: RL(P) \rightarrow Q$ so that $f \circ \eta_P = g \circ \eta_P$. (Let $e_{L(P)}$ be the left inverse of $\epsilon_{L(P)}$.) Then:
    \begin{align*}
        L(f \circ \eta_P) &= L(g \circ \eta_P)\\
        L(f) \circ e_{L(P)} \circ \epsilon_{L(P)} \circ L(\eta_P) &= L(g) \circ e_{L(P)} \circ \epsilon_{L(P)} \circ L(\eta_P)\\
        L(f) \circ e_{L(P)} &= L(g) \circ e_{L(P)}\\
        L(f) &= L(g)
    \end{align*}
    and thus $f = g$, since $L$ is faithful (as $\eta$ is pointwise monic).
\end{proof}

\begin{Lemma}\label{lem: counitfact}
Let $L: \mathcal{C} \rightarrow \mathcal{D}$ and $R: \mathcal{D} \rightarrow \mathcal{C}$ be an adjunction $L \dashv R$ with unit $\eta$ and counit $\epsilon$. Suppose $\epsilon$ is pointwise epic and $\eta R$ is pointwise split epic. Then $\epsilon$ is pointwise monic.
\end{Lemma}
\begin{proof}
    Follows dually from Lemma \ref{lem: unitfact}.
\end{proof}

In some examples, $\mathrm{Fix}_\eta = \Bmod$. In this case, the irreducibles in $\mathrm{Fix}_\epsilon$ can be identified with irreducible objects of $\Bmod$. The following proposition is useful in classifying these irreducibles.

\begin{Proposition}\label{prop: endsimple}
    Let $\mathcal{C}$ be a category, and $P \in \mathcal{C}$. There is a one to one correspondence:
    \begin{align*}
        \{F \in \mathrm{Irr}({}_\mathcal{C}\SMod):F(P) \neq 0\} &\leftrightarrow \mathrm{Irr}({}_{\End(P)}\SMod)\\
        F &\mapsto F(P)
    \end{align*}
\end{Proposition}

\begin{proof}
    See, for example, \cite[Thm.~18]{DFO1994}.
\end{proof}

\section{Examples}
\label{sec:examples}

\subsection{Single Object Case}
    If $\mathcal{B}$ is any full subcategory of ${}_A\SMod$ consisting of a single object $B$, then $\Bmod$ can be identified with ${}_S\SMod$ where $S = \End_A(B)^{\rm op}$. Under this identification, $Y=\Hom_A(B,-)$ and $|\cdot| = B \otimes_S -$.
\begin{Example}
    Suppose $\mathcal{B}=\{A\}$. Then $A \cong S$ and $|\cdot| \dashv Y$ gives the corresponding equivalence of module categories. That is, $\Fix_\epsilon = {}_A\SMod$ and $\Fix_\eta = {}_S\SMod$.
\end{Example}

\begin{Example}
    Suppose $\mathcal{B} = \{A^n\}$ for some $n \geq 1$. Then, $\Fix_\epsilon = {}_A\SMod$ and $\Fix_\eta = {}_S\SMod$.
\end{Example}

\begin{Example}
    More generally, suppose $\mathcal{B}=\{P\}$ where $P$ is a finitely generated projective generator. Then, $\Fix_\epsilon = {}_A\SMod$ and $\Fix_\eta = {}_S\SMod$. In fact, if $A_1, A_2$ are Morita equivalent $k$-algebras, with \[
    \begin{tikzcd}
        {}_{A_1}\SMod \arrow[r, shift left=.75ex, "R"] & {}_{A_2}\SMod \arrow[l, shift left=.75ex, "L"]
    \end{tikzcd}
    \]
    then there exists a finitely generated projective generator $P$ of ${}_A\SMod$ so that $A_2 \cong S$. Furthermore, under this identification, $R \cong Y$ and $L \cong |\cdot|$.
\end{Example}

For the rest of this section, let $I$ be a left ideal of an algebra $A$, and let $\mathcal{B}$ be the full subcategory of ${}_A\SMod$ with the only object $A/I$. This setting covers the examples from Section \ref{sec:intro-single-object}. We put $N=N_A(I)$, $S=N/I=(A/I)^I=\{a+I\mid a\in A,\, Ia\subset I\}$. For a left $A$-module $V$, let $V^I=\{v\in V\mid I\cdot v=0\}$.

\begin{Proposition}\label{prp:single-epsilon-free}\phantom{X}
    \begin{enumerate}[{\rm (i)}]
    \item For a left $A$-module $V$, the map $\ep_V$ is surjective if and only if $V$ is generated, as a left $A$-module, by the subspace $V^I$ of $I$-invariants.
    \item Suppose $A/I$ is free as a right $S$-module, and let $W\subset A$ be a subspace that provides a right $S$-module isomorphism $W\otimes_k S\to A/I$, $w\otimes s\mapsto (w+I)\cdot s$. Then
    \begin{equation}
        \Fix_\ep = \{V\in {}_A\SMod: \text{the map } W\otimes_k V^I\to V,\, w\otimes v\mapsto w\cdot v \text{ is bijective}\}.
    \end{equation}
    \end{enumerate}
\end{Proposition}

\begin{proof}
    (i) For a left $A$-module $V$, the map $\ep_V$ is the evaluation map
    \begin{equation}\label{eq:pf-single-epsilon}
        \ep_V:(A/I)\otimes_S\Hom_A(A/I,V)\to V,\qquad x\otimes\phi\mapsto \phi(x).
    \end{equation}
    The image of $\ep_V$ is $AV^I$, which proves the claim.
    
    (ii) We have $(A/I)\otimes_S \Hom_A(A/I,V)\cong W\otimes_k (V^I)$ and under this identification, $\ep_V$ sends $w\otimes v$ to $w\cdot v$ for $v\in V^I$.
\end{proof}

\begin{Proposition} \label{prop: singleeta}
    Suppose $(A/I)^I$ has a right $S$-module complement in $A/I$. Then $\eta$ is monic. Furthermore, 
    \begin{equation}\label{eq:complement-Fix-eta}
    \Fix_\eta = \{F \in {}_S\SMod: \Hom_A(A/I, A/I \otimes_k F) \twoheadrightarrow \Hom_A(A/I, A/I \otimes_S F)\}.
    \end{equation}
\end{Proposition}

\begin{proof}
    Let $p$ denote the projection of $A/I$ onto $S \cong (A/I)^I$. Given $F \in {}_S\SMod$, consider the map of vector spaces: $$n_F: \Hom_A(A/I, A/I \otimes_S F) \rightarrow A/I \otimes_S F \xrightarrow{p \otimes 1} S \otimes_S F \cong F.$$ This map provides a left inverse to $\eta_F$, demonstrating that $\eta$ is monic, in view of Remark \ref{rem:pointwise}.
    Consequently, $\Fix_\eta = \{F \in {}_S\SMod: \eta_F \text{ is epic}\}$. Now \eqref{eq:complement-Fix-eta} follows from the commutative diagram:
    \[\begin{tikzcd}
        F \arrow[r] & \Hom_A(A/I, A/I \otimes_S F) \\
        S \otimes_k F \arrow[u, twoheadrightarrow] \arrow[r, "\cong"] & \Hom_A(A/I, A/I \otimes_k F) \arrow[u]
    \end{tikzcd}\]
\end{proof}

\begin{Example}
    In the quantum Hamiltonian reduction case (Section \ref{sec:quantum-hamiltonian-reduction}), a vanishing of cohomology that goes back to Kostant \cite{Kos1978} in the Whittaker case is sufficient to obtain an equivalence of categories \cite[Prop.~5.1]{Tik2012}. In particular, $H^1(\Fm_\chi,U/U\Fm_\chi)=0$ implies that the inclusion $(U/U\Fm_\chi)^{\Fm_\chi}\to U/U\Fm_\chi$ splits. This coincides with the assumption of Proposition \ref{prop: singleeta}.
\end{Example}

\begin{Example}
    In the case of reduction algebras (Section \ref{sec:intro-reduction-algebras}) associated to a homomorphism $U(\Fg)\to A$, a common condition \cite{KhoOgi2008} is that for every positive root $\be$ and integer $m$, the image of the shifted coroot $h_\be+m$ is invertible in $A$, and that $\Fn_+$ acts locally nilpotently on $A$ under the adjoint action. This ensures that the extremal projector \cite{Tol2011} for $\Fg$ provides a well-defined surjective map $A/I\to S$ of right $S$-modules. The kernel of this map, which equals $\{f+I\mid f\in \Fn_- A\}$, is the required right $S$-module complement to $S$ in $A/I$. Thus, the conclusions of Proposition \ref{prop: singleeta} hold in this situation.
    Assume further that $A$ is free as a left $U(\Fg)$-module.\footnote{This holds for example if $A$ is (a localization of) an enveloping algebra of some Lie algebra containing $\Fg$, or if $\phi:U(\Fg)\to B$ is some algebra map and $A$ is (a localization of) $U(\Fg)\otimes B$ with $U(\Fg)\to A$ given by $u\mapsto (1\otimes\phi)\circ \Delta(u)$ where $\Delta$ is the comultiplication.}
    Let $M=A/I$. Then we claim that
    \begin{equation}\label{eq:reduction-algebra-example}
    U(\mathfrak{n}_-)\otimes_\C S \to M,\qquad u\otimes (n+I)\mapsto (un)+I,\quad\forall u\in U(\Fn_-), n\in N
    \end{equation}
    is an isomorphism of $(U(\mathfrak{n}_-),S)$-bimodules. In particular, $M$ is free as a right $S$-module.
    To prove this, let $m\in M$ be arbitrary. Write $m=a+I$ for some $a\in A$. Consider $n+I=P(m)$ where $P=P_\Fg$ is the extremal projector for $\Fg$. Then $a-n\in \Fn_- A$ so can be written as $a-n=\sum y_i a_i$, where $y_i\in\Fn_-$. Each $a_i$ has strictly higher weight than $a$. So by induction, each $a_i$ is in the image of \eqref{eq:reduction-algebra-example}. Therefore $a$ is also in the image. This shows that the map is surjective. That the map is injective follows from the fact that $U(\Fn_-)$ acts injectively on $M$. That is, $\ann_{U(\Fn_-)} m=0$ for all nonzero $m\in M$.

    Thus, by Proposition \ref{prp:single-epsilon-free}, we have a precise description of the subcategory $\Fix_\ep$ of ${}_A\SMod$. It is the category of left $A$-modules $V$ which (i) are generated by its subspace $V^+$ of $\Fn_+$-highest weight vectors, and (ii)  $U(\Fn_-)$ acts injectively on $V^+$. Equivalently, the map $U(\Fn^-)\otimes V^+\to V$, $u\otimes v\mapsto u\cdot v$ is an isomorphism of left $U(\Fn_-)$-modules.
\end{Example}

We end this subsection with some consequences when $A/I$ is a simple $A$-module.

\begin{Proposition} \label{prop: singleepsilon}
    If $A/I$ is simple as a left $A$-module, then $\epsilon$ is monic.
\end{Proposition}
\begin{proof}
     Since $A/I$ is simple, $S$ is a division algebra, and $A/I$ is a free right $S$-module. Let $\{b_i\}$ be an $S$-basis for $A/I$. Let $V \in {}_A\SMod$, and $\sum_i b_i \otimes \varphi_i \in A/I \otimes_S \Hom_A(A/I, V)$ so that $\epsilon_V(\sum_i b_i \otimes \varphi_i) = 0$. By the Jacobson Density Theorem, for each $i$, there exist $a_j \in A$ so that $a_j.b_i = \delta_{i,j} + I$ for each $j$. Thus for each $j$, we have $0 = a_j.\epsilon_V(\sum_i b_i \otimes \varphi_i) = \epsilon_V(1 + I \otimes \varphi_j) = \varphi_j(1 + I)$. So $\varphi_j = 0$ for each $j$, and thus $\sum_i b_i \otimes \varphi_i = 0$. So $\epsilon$ is monic.
\end{proof}

\begin{Corollary}\label{cor:simple-case}
    If $A/I$ is a simple $A$-module, the following statements hold.
    \begin{enumerate}[{\rm (i)}]
        \item $\Fix_\epsilon = \{V \in {}_A\SMod: A.V^I = V\}$.
        \item $\Fix_\eta = {}_S\SMod$.
    \end{enumerate}
\end{Corollary}
\begin{proof}
    By Proposition \ref{prop: singleepsilon}, $\epsilon$ is monic. Therefore, $\Fix_\epsilon = \{V \in {}_A\SMod: \epsilon_V \text{ is epic}\} = \{V \in {}_A\SMod: A.V^I = V\}$.

    Since $A/I$ is simple, $S$ is a division algebra by Schur's Lemma, so $(A/I)^I \cong S$ has a right $S$-module complement in $A/I$ (simply extend a basis). By Proposition \ref{prop: singleeta}, $\eta$ is monic.
    Now the triangle identity $(\epsilon|\cdot|) \circ (|\cdot|\eta) = 1$ shows that $\epsilon_{|F|}$ is an epimorphism for every $F \in {}_S\SMod$, so that $|\cdot|$ maps ${}_S\SMod$ into $\Fix_\epsilon$. So $|\cdot| \dashv Y$ restricts to an adjunction:
     \[
    \begin{tikzcd}
        \Fix_\epsilon \arrow[r, shift left=.75ex, "Y"] & {}_S\SMod \arrow[l, shift left=.75ex, "|\cdot|"]
    \end{tikzcd}
    \]
    By Lemma \ref{lem: unitfact}, $\eta$ is an isomorphism. Thus $\Fix_\eta = {}_S\SMod$.
\end{proof}

\subsection{Many Objects; Weight Modules}
Let $X$ be a set of left ideals of an algebra $A$ so that $I + J = A$ for distinct $I,J \in X$. Then $\bigoplus_{I \in X} \Hom_A(A/I,V) \rightarrow V$ is an injective linear map for each $V \in {}_A\SMod$ (that is, the sum of weight spaces of $V$ is direct). We call $V \in {}_A\SMod$ a \emph{weight module}, if $\bigoplus_{I \in X} \Hom_A(A/I,V) \rightarrow V$ is an isomorphism. We denote by $\mathsf{Wt}(A, X)$ the full subcategory of ${}_A\SMod$ consisting of weight modules.

\begin{Lemma}\label{lem: weightcategory}
    $\mathsf{Wt}(A, X)$ is closed under taking submodules and quotients. Thus $V \in \mathsf{Wt}(A,X)$ is simple if and only if it is simple as an $A$-module. Furthermore, given $I \in X$, if $A/I$ is a weight module, then any simple $A$-module $V$ with $V^I \neq 0$ is a weight module.
\end{Lemma}

Let $\mathcal{B}$ be a full subcategory of ${}_A\SMod$ with object set $\mathcal{B}=\{A/I\}_{I \in X}$. We have the adjunction $|\cdot| \dashv Y$, with unit $\eta$ and counit $\epsilon$. We show that $\mathsf{Wt}(A,X)$ arises from this adjunction.

\begin{Proposition} \label{prop: weightfix}
    If $M$ is a weight module for each $M \in \mathcal{B}$, then
    \begin{enumerate}[{\rm (i)}]
        \item $\Fix_\epsilon = \mathsf{Wt}(A,X).$
        \item $\Fix_\eta = \Bmod.$
    \end{enumerate}
    In particular, the category $\mathsf{Wt}(A,X)$ of weight modules over $A$ is equivalent to $\Bmod$.
\end{Proposition}
\begin{proof}
    Let $F \in \Bmod$. For each $P \in \mathcal{B}$, the evaluation map $\Hom_A(P,N) \otimes F(N) \rightarrow F(P)$ given by $f \otimes v \mapsto F(f)(v)$ gives us a map in the following commutative diagram:
    \[
    \begin{tikzcd}
        \bigoplus_{M,N} \Hom_A(M,N) \otimes F(N) \arrow[r] & \bigoplus_M F(M)\\
        \bigoplus_N \Hom_A(P,N) \otimes F(N) \arrow[r] \arrow[u] & F(P) \arrow[u]
    \end{tikzcd}
    \]
    Let $n \in N_1, g \in \Hom_A(N_1, N_2), v \in F(N_2)$. Say $n \mapsto \sum_M f_M$ under the natural isomorphism $N_1 \rightarrow \bigoplus_M \Hom_A(M, N_1)$. Then naturality implies that $g(n) \mapsto \sum_M g \circ f_M$. So we have a map:
    $$\omega_F: \bigoplus_N N \otimes F(N) \rightarrow \bigoplus_{M,N} \Hom_A(M,N) \otimes F(N) \rightarrow \bigoplus_M F(M)$$
    such that
    \begin{align*}
    g(n) \otimes v - n \otimes F(g)(v) &\mapsto \sum_M g \circ f_M \otimes v - \sum_M f_M \otimes F(g)(v)\\
    &\mapsto \sum_M F(g \circ f_M)(v) - F(f_M)F(g)(v)\\
    &= 0
    \end{align*}
    So there is an induced map $n_F:|F| \rightarrow \bigoplus_M F(M)$ which makes the following diagram commute:
    \[
    \begin{tikzcd}
        \bigoplus_N N \otimes F(N) \arrow[r, twoheadrightarrow] \arrow[d, "\omega_F"] & {|F|} \arrow[dl, dashed, "n_F"]\\
        \bigoplus_M F(M)
    \end{tikzcd}
    \]
    This map provides an inverse to the linear map $\bigoplus_M F(M) \rightarrow |F|$ (which we call, by abuse of notation, $\eta_F$) defined by the commutative diagram:
    \[\begin{tikzcd}
        \bigoplus_M F(M) \arrow[r, "\eta_F"] & {|F|}\\
        F(M) \arrow[r, "(\eta_F)_M"] \arrow[u] & \Hom_A(M, {|F|}) \arrow[u]
    \end{tikzcd}\]
    In particular, this implies that $\bigoplus_M \Hom_A(M, |F|) \rightarrow |F|$ is an isomorphism. That is, $|F|$ is a weight module. Furthermore, each $(\eta_F)_M$ is an isomorphism, so that the weight space $\Hom_A(M, |F|)$ can be identified with $F(M)$. Thus $\Fix_\eta = \Bmod$.
    
    Furthermore, if we take $F = Y(V)$ for any $V \in {}_A\SMod$, we have the commutative diagram:
    \[\begin{tikzcd}
        \bigoplus_M \Hom_A(M, V) \arrow[r] \arrow[d, "\cong"] & V\\
        {|Y(V)|} \arrow[ur, "\epsilon_V", swap]
    \end{tikzcd}\]
    so that $\epsilon_V$ is an isomorphism precisely when $V$ is a weight module.
\end{proof}

As a corollary, we recover the following:
\begin{Corollary}[{\cite[Thm.~4.9]{LepMcC1973}}]
    If $M$ is a weight module for each $M \in \mathcal{B}$, then for each $I \in X$ there is a one to one correspondence: \begin{align*}\{V \in \mathrm{Irr}({}_A\SMod):\Hom_A(A/I,V) \neq 0\} &\leftrightarrow \mathrm{Irr}({}_{\End_A(A/I)}\SMod)\\
    V &\mapsto \Hom_A(A/I,V)\end{align*}
\end{Corollary}

\begin{proof}
    Follows from Lemma \ref{lem: weightcategory}, Proposition \ref{prop: weightfix}, and Proposition \ref{prop: endsimple}.
\end{proof}

\begin{Example}
If $A=A(R,\si,t)=\oplus_{g\in\Z^n} A_g$ is a (twisted) generalized Weyl algebra of degree $n$ with base ring $R$, then we can take $X$ to be the set of all left ideals of $A$ of the form $A\Fm$ where $\Fm$ ranges over the maximal spectrum of $R$. Then we recover the result of \cite{MazPonTur2003} stating that, for fixed $\Fm$, the irreducible $A$-modules $V$ that are weight modules with respect to $R$ and whose weight space $V_\Fm$ is nonzero, are in bijection with the simple $C(\Fm)/\langle\Fm\rangle$-modules, where $C(\Fm)=\oplus_{g\in G_\Fm} A_g$ and $G_\Fm=\{g\in\Z^n\mid \si_1^{g_1}\cdots\si_n^{g_n}(\Fm)=\Fm\}$ is the stabilizer subgroup of $\Z^n$ at $\Fm$, and $\langle\Fm\rangle$ is the two-sided ideal in $C(\Fm)$ generated by $\Fm$.
More generally, we may take any $\Z^n$-stable subset $S$ of the prime spectrum of $R$ with the property that $\mathfrak{p}+\mathfrak{q}=R$ for any distinct $\mathfrak{p},\mathfrak{q}\in S$ and let $X=\{A\mathfrak{p}\mid \mathfrak{p}\in S\}$.
\end{Example}

\begin{Example}
Let $\Fg$ be a simple Lie algebra with Cartan subalgebra $\Fh$. For each $\la\in\Fh^\ast$, let $\Fm_\la$ be the kernel of the evaluation homomorphism $U(\Fh)\to\C$ given by $h\mapsto\la(h)$. Let $X=\{A\Fm_\la\mid\la\in\Fh^\ast\}$. Then we recover the basic fact about weight modules over Lie algebras, that they are equivalent to modules over the centralizer $U(\Fg)^\Fh$. More precisely, for fixed $\mu\in\Fh^\ast$, (isoclasses of) simple $\Fh$-weight $\Fg$-modules $V=\oplus_{\la\in\Fh^\ast}V[\la]$ with $V[\mu]\neq 0$ are in bijection with (isoclasses of) simple modules over the centralizer $U(\Fg)^\Fh$. This restricts to weight modules with finite-dimensional weight spaces and finite-dimensional modules over the centralizer.
\end{Example}

\subsection{Generalized Weight Modules}

Let $A$ be an algebra, and let $\mathbb{X}$ be a set of downward directed posets of left ideals of $A$. The elements of $\mathbb{X}$ are called \emph{blocks}, and each block $\I \in \mathbb{X}$ is considered to be a poset category whose morphisms are the inclusion maps. We assume $I + J = A$ for any $I,J$ from distinct blocks of $\mathbb{X}$.

Given $V \in {}_A\SMod$ and $\I \in \mathbb{X}$, the {\em generalized weight space} of $V$ corresponding to $\I$ is $V(\I) = \{v \in V: Iv = 0 \text{ for some } I \in \I\}$. Since, $\I$ is downward directed, $V(\I)$ is the directed colimit $\varinjlim V^I= \colim_{I \in \I} V^I \cong \colim_{I \in \I}\Hom_A(A/I,V)$.

Now $\bigoplus_{\I \in \mathbb{X}} \colim_{I \in \I} \Hom_A(A/I,V) \rightarrow V$ is an injective linear map for each $V \in {}_A\SMod$ (that is, the sum of generalized weight spaces of $V$ is direct). We call $V \in {}_A\SMod$ a \emph{generalized weight module}, if $\bigoplus_{\I \in \mathbb{X}} \colim_{I \in \I} \Hom_A(A/I,V) \rightarrow V$ is an isomorphism. We denote by $\HC$ the full subcategory of ${}_A\SMod$ consisting of generalized weight modules.

\begin{Example}
    Let $A$ be an algebra with subalgebra $\Gamma$. The {\em cofinite spectrum} $\mathrm{cfs}(\Gamma)$ of $\Gamma$ is the set of maximal two-sided ideals of $\Gamma$ of finite codimension. In \cite{DFO1994}, the algebra $\Gamma$ was assumed to be quasicommutative. By definition, this means $\Ext_\Ga^1(S_\Fm,S_\Fn)=0$ for all $\Fm\neq\Fn\in\mathrm{cfs}(\Ga)$, where $S_\Fm$ is the unique simple $\Ga$-module with annihilator $\Fm$.
    For each $\frak{m} \in \mathrm{cfs}(\Gamma)$, let $\I(\frak{m}) = \{A\frak{m}^k: k\geq 0\}$. Let $\mathbb{X} = \{\I(\frak{m}):\frak{m} \in \mathrm{cfs}(\Gamma)\}$. Then $\HC$ is the category of Harish-Chandra modules from \cite{DFO1994}.
\end{Example}

\begin{Example}
    More generally, we may obtain the category of block modules, introduced in \cite{Fil2024}. Again, let $A$ be an algebra with subalgebra $\Gamma$. Let $\sim$ be an equivalence relation on $\mathrm{cfs}(\Gamma)$ (which is typically given by the failure of $\Gamma$ to be quasicommutative). For each $\frak{m} \in \mathrm{cfs}(\Gamma)$, let $\I([\frak{m}]_\sim) = \{A\frak{m}_1\cdots\frak{m}_k: \frak{m}_1,\ldots,\frak{m}_k \in [\frak{m}]_\sim\}$. Let $\mathbb{X} = \{\I([\frak{m}]_\sim):\frak{m} \in \mathrm{cfs}(\Gamma)\}$. Then $\HC$ is the category of Harish-Chandra block modules with respect to $\sim$.
\end{Example}

\begin{Lemma}\label{lem: genweightcategory}
    $\HC$ is closed under taking direct sums, submodules, quotients. Thus $V \in \HC$ is simple if and only if it is simple as an $A$-module. Furthermore, given $I \in \mathbb{X}$, if $A/I$ is a generalized weight module, then any simple $A$-module $V$ with $V^I \neq 0$ is a generalized weight module.
\end{Lemma}

Let $\mathcal{B}$ be a full subcategory of ${}_A\SMod$ with object set $\mathcal{B}=\{A/I\}_{I \in \I \in \mathbb{X}}$. We have the adjunction $|\cdot| \dashv Y$, with unit $\eta$ and counit $\epsilon$. We show that $\HC$ arises from this adjunction.

\begin{Proposition}\label{prop: genfix}
    If $M$ is a generalized weight module for each $M \in \mathcal{B}$, the following statements hold:
    \begin{enumerate}[{\rm (i)}]
        \item $\mathrm{Fix}_\epsilon = \HC$
        \item $\mathrm{Ob}(\mathrm{Fix}_\eta) = \{F \in \Bmod: F \cong Y(V) \text{ for some } V \in {}_A\SMod\}$
    \end{enumerate}
\end{Proposition}
\begin{proof}
    We begin analogously to the proof of Proposition \ref{prop: weightfix}. Let $F \in \Bmod$. The evaluation maps $\Hom_A(A/I,N) \otimes F(N) \rightarrow F(A/I)$ given by $f \otimes v \mapsto F(f)(v)$ gives us maps in the following commutative diagram:
    \[
    \begin{tikzcd}
        \bigoplus_N \bigoplus_\I \colim_{I \in \I} \Hom_A(A/I,N) \otimes F(N) \arrow[r] & \bigoplus_\I \colim_{I \in \I} F(A/I)\\
        \bigoplus_N \colim_{I \in \I} \Hom_A(A/I,N) \otimes F(N) \arrow[r] \arrow[u] & \colim_{I \in \I} F(A/I) \arrow[u]
    \end{tikzcd}
    \]
    Let $n \in N_1, g \in \Hom_A(N_1, N_2), v \in F(N_2)$. Say $n \mapsto \sum_\I \iota_(f_\I)$ under the natural isomorphism $N_1 \rightarrow \bigoplus_\I \colim_I \Hom_A(A/I, N_1)$. Then naturality implies that $g(n) \mapsto \sum_\I \iota(g \circ f_\I)$. So we have a map:
    $$\omega_F: \bigoplus_N N \otimes F(N) \rightarrow \bigoplus_{N}\bigoplus_\I \colim_I \Hom_A(A/I,N) \otimes F(N) \rightarrow \bigoplus_\I \colim_I F(A/I)$$
    such that
    \begin{align*}
    g(n) \otimes v - n \otimes F(g)(v) &\mapsto \sum_\I \iota(g \circ f_\I) \otimes v - \sum_\I \iota(f_\I) \otimes F(g)(v)\\
    &\mapsto \sum_\I F(g \circ f_\I)(v) - F(f_\I)F(g)(v)\\
    &= 0
    \end{align*}
    So there is an induced map $n_F:|F| \rightarrow \bigoplus_\I \colim_I F(A/I)$ which makes the following diagram commute:
    \[
    \begin{tikzcd}
        \bigoplus_N N \otimes F(N) \arrow[r, twoheadrightarrow] \arrow[d, "\omega_F"] & {|F|} \arrow[dl, dashed, "n_F"]\\
        \bigoplus_\I \colim_I F(A/I)
    \end{tikzcd}
    \]
    This map provides an inverse to the linear map $\bigoplus_\I \colim_I F(A/I) \rightarrow |F|$ (which we call, by abuse of notation, $\eta_F$) defined by the commutative diagram:
    \[\begin{tikzcd}
        \bigoplus_\I \colim_I F(A/I) \arrow[r, "\eta_F"] & {|F|}\\
        \colim_I F(A/I) \arrow[r, "\colim_I (\eta_F)_{A/I}"] \arrow[u] & \colim_I \Hom_A(A/I, {|F|}) \arrow[u]
    \end{tikzcd}\]
    
    In particular, this implies that $\bigoplus_\I \colim_I \Hom_A(A/I,|F|) \rightarrow |F|$ is an isomorphism. That is, $|F|$ is a generalized weight module. Furthermore, each $\colim_{I \in \I} (\eta_F)_{A/I}$ is an isomorphism, so that the generalized weight space $\colim_{I\in \I} \Hom_A(A/I, |F|)$ can be identified with $\colim_{I\in \I} F(A/I)$.

    Furthermore, if we take $F = Y(V)$ for any $V \in {}_A\SMod$, we have the commutative diagram:
    \[\begin{tikzcd}
        \bigoplus_\I \colim_{I} \Hom_A(A/I, V) \arrow[r] \arrow[d, "\cong"] & V\\
        {|Y(V)|} \arrow[ur, "\epsilon_V", swap]
    \end{tikzcd}\]
    so that $\epsilon_V$ is an isomorphism precisely when $V$ is a generalized weight module. Thus $\mathrm{Fix}_\epsilon = \HC$.

    Since $\epsilon |\cdot|$ is a natural isomorphism, $\eta Y$ is also an isomorphism \cite{MacDSt1982}. Thus $\mathrm{Ob}(\mathrm{Fix}_\eta) = \{F \in \Bmod: F \cong Y(V) \text{ for some } V \in {}_A\SMod\}$.
    \end{proof}

    \begin{Proposition}
        If $M$ is a generalized weight module for each $M \in \mathcal{B}$, then $\Fix_\eta$ is an abelian subcategory of $\Bmod$.
    \end{Proposition}
    \begin{proof}
        By Proposition \ref{prop: genfix}, $\mathrm{Fix}_\epsilon = \HC$. Thus by \ref{lem: genweightcategory}, $\mathrm{Fix}_\epsilon$ is an abelian subcategory of ${}_A\SMod$ and the inclusion functor $I_\epsilon: \mathrm{Fix}_\epsilon \rightarrow {}_A\SMod$ is exact. The inclusion $I_\eta: \mathrm{Fix}_\eta \rightarrow \Bmod$ is naturally isomorphic to $Y \circ I_\epsilon \circ |\cdot||_{\mathrm{Fix}_\eta}$. Since $Y$ is a right adjoint, $Y$ is left exact. Since $|\cdot| \dashv Y$ restricts to an equivalence between $\mathrm{Fix}_\epsilon$ and $\mathrm{Fix}_\eta$, $|\cdot||_{\mathrm{Fix}_\eta}$ is exact. Thus $I_\eta$ is left exact. Additionally, since $\epsilon |\cdot|$ is a natural isomorphism, $I_\eta$ is left adjoint to $Y|\cdot|: \Bmod \rightarrow \Fix_\eta$ \cite{Gran2021}, and thus right exact.
    \end{proof}
        
    Now $\HC$ is equivalent to a full abelian subcategory of $\Bmod$. In order to apply Proposition \ref{prop: endsimple} and quantify the irreducibles, we would need to show irreducibility in $\Fix_\eta$ is the same as irreducibility in $\Bmod$. Alternatively, sufficient conditions for the finiteness of $\mathrm{Irr}(\HC)$ were given in \cite{DFO1994} for the case of Harish-Chandra modules, and more generally in \cite{Fil2024} for the case of block modules. We note, however, that if an $A$-module $V$ is equipped with the discrete topology, then $$V(\I) \cong \mathrm{Cont}_A(\lim_{I\in \I}A/I, V).$$ This suggests that a topologically enriched approach to restricted Yoneda embeddings may generalize these results and provide the appropriate context for answering questions about generalized weight modules.

\end{document}